\documentclass[12pt, reqno]{amsart}
\usepackage{amsmath, amsthm, amscd, amsfonts, amssymb, graphicx, color}
\usepackage{amssymb,amsmath,amsfonts,latexsym}
\usepackage{bm}
\usepackage[all,cmtip]{xy}
\usepackage{amscd}
\DeclareMathOperator*{\Mprod}{\text{\raisebox{0.25ex}{\scalebox{0.8}{$\prod$}}}}

\newtheorem{theorem}{Theorem}[section]
\newtheorem{lemma}[theorem]{Lemma}
\newtheorem{proposition}[theorem]{Proposition}
\newtheorem{corollary}[theorem]{Corollary}
\theoremstyle{definition}
\newtheorem{definition}[theorem]{Definition}
\newtheorem{example}[theorem]{Example}

\theoremstyle{remark}
\newtheorem{remark}[theorem]{Remark}
\numberwithin{equation}{section}

\usepackage{bm}

\begin{document}


\setcounter{page}{1}

\title[Dynamics]{linear dynamics of the adjoint of a Unilateral {weighted shift operator}}

\author[Das]{Bibhash Kumar Das}
\address{Indian Institute of Technology Bhubaneswar, Jatni Rd, Khordha - 752050, India}
\email{bkd11@iitbbs.ac.in}

\author[Mundayadan]{Aneesh Mundayadan}
\address{Indian Institute of Technology Bhubaneswar, Jatni Rd, Khordha - 752050, India}
\email{aneesh@iitbbs.ac.in}



\begin{abstract}
This paper is a sequel to our work in \cite{Das-Mundayadan}. Here, we primarily study the dynamics of the adjoint of a  weighted forward shift operator $F_w$ on the analytic function space $\ell^p_{a,b}$ having a normalized Schauder basis of the form $\{(a_n+b_nz)z^n:~n \geq 0\}$. We obtain sufficient conditions for $F_w$ to be continuous, and show, under certain conditions, that the operator $F_w$ is similar to a compact perturbation of a weighted forward shift on $\ell^p(\mathbb{N}_0)$. This also allows us to obtain the essential spectrum of $F_w$. Further, we study when the adjoint $F_w^*$ is hypercyclic, mixing, and chaotic, and provide a class of chaotic operators that are compact perturbations of weighted shifts on $\ell^p(\mathbb{N}_0)$. Finally, it is proved that the adjoint of a shift on the dual of $\ell^p_{a,b}$ can have non-trivial periodic vectors, without being even hypercyclic. Also, the zero-one law of orbital limit points fails for $F_w^*$, which means that, under certain conditions, the adjoint $F_w^*$ is non-hypercyclic, but it has an orbit possessing non-zero norm limit points.
\end{abstract}
\maketitle

\noindent
\tableofcontents
\section{Introduction}

Among various classes of operators that are commonly investigated in operator theory and linear dynamics, the class of weighted shifts has a predominant role. Shift operators are fairly simple to define, yet they yield diverse and appealing results. In operator theory, weighted shifts are well studied in connections with model theory (representation of various operators using shifts), invariant subspaces, cyclic vectors, subnormality, and so on, cf. Halmos \cite{Halmos} and Shields \cite{Shields}. In the context of linear dynamics it is a usual practice that dynamical properties are first tested for weighted shifts defined on sequence spaces.  It is always of interest to find new spaces on which weighted shifts can be investigated. In \cite{Das-Mundayadan}, the authors introduced the spaces $\ell^p_{a,b}$ and $c_{0,a,b}$, and studied the dynamics of weighted backward shifts acting on them. See the section $2$ below for these spaces. In this paper, as a follow up to \cite{Das-Mundayadan} we aim to study and highlight the similarities and differences in the dynamics of the adjoint $F_w^*$ of the weighted forward shift operator $F_w$ acting on $\ell^p_{a,b}$ and $c_{0,a,b}$. Let us recall the relevant notions from operator dynamics.

 \begin{definition}
 An operator $T$ on a separable Banach space $X$ is said to be
 \begin{itemize}
 \item \textit{hypercyclic} if there exists $x\in X$, known as a \textit{hypercyclic vector} for $T$, such that the orbit $\{x,Tx,T^2x,\cdots\}$ is dense in $X$,
\item \textit{chaotic} if $T$ is hypercyclic and the set $\{v\in X:T^mv=v,~for ~some~m\in \mathbb{N}\}$ of periodic vectors is dense in $X$,
\item \textit{topologically transitive} if, for two non-empty open sets $U_1$ and $U_2$ of $X$, there exists a natural number $k$ such that $T^k(U_1)\cap U_2\neq \phi$, 
\item \textit{topologically mixing} on $X$ if, for any two non-empty open sets $U_1$ and $U_2$ of $X$, there exists $N$, a natural number, such that $T^n(U_1)\cap U_2\neq \phi$ for all $n\geq N$. 
 \end{itemize}
 \end{definition}
\noindent We recall that the periodic vectors for an operator $T$ defined on a complex vector space are exactly the linear combinations of eigenvectors corresponding to ``rational" unimodular eigenvalues. For an account on the theory of hypercyclicity and related topics, we refer to the monographs by F. Bayart and E. Matheron \cite{Bayart-Matheron}, and K.-G. Grosse-Erdmann and A. Peris \cite{Erdmann-Peris}. Also, see S. Grivaux et al. \cite{Grivaux etal}.

Rolewicz \cite{Rolewicz} showed that $\lambda B$ is hypercyclic on $\ell^p(\mathbb{N}_0)$, $1\leq p<\infty$, where $|\lambda|>1$ and $B$ is the unweighted backward shift operator. Also, for particular classes of hypercyclic weighted shifts, see Kitai \cite{Kitai}, Gethner and Shapiro \cite{Gethner-Shapiro}, and Godefroy and Shapiro \cite{Godefroy-Shapiro}. A complete characterization of hypercyclic weighted shifts on $\ell^p(\mathbb{N}_{0})$ and $\ell^p(\mathbb{Z})$ was obtained by Salas \cite{Salas-hc}. In \cite{Godefroy-Shapiro} the authors introduced and studied the notion of chaos for operators, and established that several familiar operators including weighted shifts are chaotic. A complete characterization of hypercyclic and chaotic weighted shifts in a general sequence $F$-space was obtained by Grosse-Erdmann \cite{Erdmann}. We also note that the weighted shifts which are (topologically) mixing were characterized by Costakis and Sambarino \cite{Costakis}. In practice, the existence of hypercyclic and chaotic properties of operators is established using well known criteria that are stated below, cf. \cite{Bayart-Matheron} and \cite{Erdmann-Peris}. We will make use of these criteria in Section $3$.

\begin{theorem}\label{thm-hypc} $\emph{\textsf{(Gethner-Shapiro Criterion~\cite{Gethner-Shapiro})}}$
Let $T$ be a bounded operator on a separable Banach space $X$, and let $D$ be a dense subset of $X$. If $\{n_k\} \subseteq \mathbb{N}$ is a strictly
increasing sequence and $S:D\rightarrow D$ is a map such that, for each $x\in D$, $\lim_{k\rightarrow \infty} T^{n_k}x=0=\lim_{k\rightarrow \infty}S^{n_k}x,$ and $TSx=x$, then $T$ is hypercyclic. In addition, if $n_k=k$ for all $k\geq 1$, then $T$ is mixing on $X$.
\end{theorem}
A similar criterion, known as the chaoticity criterion, has been used to obtain chaotic operators in $F$-spaces, and we refer to \cite{Bonilla-Erdmann1} where it is originally stated. This criterion is very strong: indeed, if an operator $T$ satisfies the chaoticity criterion, then $T$ is topologically mixing, chaotic, frequently hypercyclic, and it has invariant measures of full support, cf. \cite{Bayart-Matheron}.

\begin{theorem}\emph{\textsf{(Chaoticity Criterion \cite{Bonilla-Erdmann1})}} \label{chaos}
Let $X$ be a separable Banach space, $D$ be a dense set in $X$, and let $T$ be a bounded operator on $X$. If there exists a map $S:D\rightarrow D$ such that $\sum_{n\geq 0} T^nx$ and $\sum_{n\geq 0} S^nx$ are unconditionally convergent, and $TSx=x$,
for each $x\in D$, then the operator $T$ is chaotic and mixing on $X$. 
\end{theorem}

The paper is organized as follows. In Section $2$, we recall the spaces $\ell^p_{a,b}$ and $c_{0,a,b}$, and study the continuity of the weighted forward shift $F_w$ on these spaces. Particularly, we obtain an important ``shift-like" property of the adjoint $F_w^*$ although the adjoint is not necessarily a weighted backward shift. In Section $3,$ as main results, we study when $F_w^*$ is hypercyclic, mixing, and chaotic. It is shown that forward shifts on $\ell^p_{a,b}$ and $c_{0,a,b}$ have distinguishing properties compared to classical weighted shifts. For example, among other results, it is shown that the adjoint $F^*$ of the unweighted forward shift $F$ can admit a non-zero periodic vector, without being hypercyclic. Also, the zero-one law of hypercyclicity fails for $F^*$.

\section{The weighted forward shift $F_w$ on the spaces $\ell^p_{a,b}$ and $c_{0,a,b}$}

In this section we recall the spaces $\ell^p_{a,b}$ and $c_{0,a,b}$ from \cite{Das-Mundayadan}, and as main results, we provide conditions to ensure the continuity of a weighted forward shift $F_w$ on these spaces. Also, an important shifting property of the adjoint $F_w^*$ will be established (cf. Proposition \ref{uni-shift}).

Adams and McGuire \cite{Adams-McGuire}, motivated by some operator-theoretic questions, studied the analytic reproducing kernel Hilbert space $\ell^2_{a,b}$ having an orthonormal basis of the form $(a_n+b_nz)z^n$, $n\geq 0$. As generalizations, the authors of \cite{Das-Mundayadan} introduced the analytic function spaces $\ell^p_{a,b}$ and $c_{0,a,b}$, and studied the dynamics of a weighted backward shift $B_w$ on these spaces. Let $a:=\{a_n\}_{n=0}^{\infty}$ and $b:=\{b_n\}_{n=0}^{\infty}$ be complex sequences, $a_n\neq 0$, and $b_n\neq 0$, for all $n\geq 0$. The authors of \cite{Das-Mundayadan} considered $\ell^p_{a,b}$ to be the space of all analytic functions, having a normalized Schauder basis $\{f_n\}_{n=0}^{\infty}$, where $$f_n(z):=(a_n+b_nz)z^n \hskip 1cm (n\geq 0),$$ and it is assumed to be equivalent to the standard ordered basis in the sequence space $\ell^p(\mathbb{N}_{0})$, $1\leq p < \infty$, that is,  
\begin{equation}\label{Laurent}
f(z)=\sum_{n=0}^{\infty}\lambda_{n}f_{n}(z) \in \ell^p_{a,b},
\end{equation}
if and only if
\begin{equation}\label{norm-f}
\lVert f \rVert_{\ell^p_{a,b}}:=\Big(\sum_{n=0}^{\infty}\lvert\lambda_{n}\rvert^{p}\Big)^{\frac{1}{p}}<\infty.
\end{equation}
In a similar way, the space $c_{0,a,b}$ was introduced by replacing \eqref{norm-f} with $\|f\|_{c_{0,a,b}} := \sup_{n \ge 0} |\lambda_{n}|,$ see the next page. Recall that (cf. Lindenstrauss and Tzafriri \cite{Lindenstrauss}, p. $5$) a sequence $\{u_n\}_{n=0}^{\infty}$ in a Banach space $X$ is called a Schauder basis if each vector $u\in X$ has a norm-convergent expansion $u=\sum_{n=0}^{\infty} \lambda_n u_n$ for a unique scalar sequence $\{\lambda_n\}$. Also, for given Banach spaces $X$ and $Y$, we say that the Schauder bases $\{u_{n}\}_{n=0}^{\infty}$ of $X$ and $\{v_{n}\}_{n=0}^{\infty}$ of $Y$, are equivalent if the convergence of $\sum_{n=0}^{\infty}\lambda_{n}u_{n}$ in $X$ is equivalent to that of $\sum_{n=0}^{\infty}\lambda_{n}v_{n}$ in $Y$.  

Since the existence of $\ell^p_{a,b}$ and $c_{0,a,b}$ was not mentioned explicitly in \cite{Das-Mundayadan}, below we provide some justifications. As in \cite{Das-Mundayadan}, given two sequences $a:=\{a_n\}_{n=0}^{\infty}$ and $b:=\{b_n\}_{n=0}^{\infty}$ of non-zero complex numbers, let 
\begin{center}
    $\mathcal{D}_{1}:=\left\{ z \in \mathbb{C}:\sum_{n=0}^{\infty} \big(( \lvert a_{n}\rvert+\lvert b_{n}\rvert)\lvert z \rvert^{n} \big)^{q}<\infty\right\}$,
    \end{center}
    where $1<p<\infty$ and $\frac{1}{p}+\frac{1}{q}=1,$
     \begin{center}
    $\mathcal{D}_{2}:=\left\{ z \in \mathbb{C}:\sup_{n\geq 0} ~(\lvert a_{n}\rvert+\lvert b_{n}\rvert)\lvert z\rvert^{n}<\infty\right\},$
    \end{center}
    and
    \begin{center}
    $\mathcal{D}_{3}:=\left\{ z \in \mathbb{C}:\sum_{n=0}^{\infty} ( \lvert a_{n}\rvert+\lvert b_{n}\rvert)\lvert z \rvert^{n} <\infty\right\}.$
    \end{center}
Let $R$ be given by
\[
\frac{1}{R}=\limsup_{n\rightarrow \infty} ~(|a_n|+|b_n|)^{1/n}.
\]
\textit{To avoid the degenerate cases, throughout the paper we let $\limsup_{n\rightarrow \infty}(|a_n|+|b_n|)^{1/n}<\infty$.}

Note that the disc $\mathbb{D}_R:=\{z\in \mathbb{C}:|z|<R\}$ is contained in $\mathcal{D}_1$, $\mathcal{D}_2$, and $\mathcal{D}_3$. Consider, first, the (open or closed) disc $\mathcal{D}_1$ for the case $1<p<\infty$, and we will define $\ell^p_{a,b}$. For a compact set $K$ in $\mathcal{D}_1$, $z \in K$, and a sequence $\{\lambda_n\}_{n=0}^{\infty}\in \ell^p(\mathbb{N}_{0})$, we note that

 \begin{equation}\label{series}
    \sum_{n=0}^{\infty} \lvert \lambda_{n}|(|a_{n}|+|b_{n}z|)|z|^{n} \leq \max_{z\in K} \{1, \lvert z\rvert\} \Big( \sum_{n=0}^{\infty}(\lvert a_{n}\rvert+\lvert b_{n}\rvert)^{q}\lvert z \rvert^{qn}\Big)^{\frac{1}{q}} \Big(\sum_{n=0}^{\infty} \lvert\lambda_{n}\rvert^p\Big)^{\frac{1}{p}},
\end{equation}
which is finite. We see that $f(z):=\sum_{n=0}^{\infty} \lambda_{n}(a_{n}+b_{n}z)z^{n}$ defines an analytic function on the open disc $\mathbb{D}_R:=\{z\in \mathbb{C}:|z|<R\}$ because the series converges uniformly and absolutely on a compact set in $\mathbb{D}_R$. Now, we ensure that $f(z)$ has a unique representation with respect to \eqref{Laurent} and \eqref{norm-f}. The uniqueness follows since, if $f(z)=\sum_{n=0}^{\infty} \lambda_{n}(a_{n}+b_{n}z)z^{n}=\sum_{n=0}^{\infty} \mu_{n}(a_{n}+b_{n}z)z^{n}$ for $\{\lambda_n\}$ and $\{\mu_n\}$ in $\ell^p(\mathbb{N}_0)$, and $|z|<R$, then at $z=0$, we must have $\lambda_0=\mu_0$, and also, the first (and higher order) differentiation gives $\lambda_1=\mu_1$, $\lambda_2=\mu_2$, and so on, where we have used the assumption that $a_n$ is non-zero for all $n$. Thus, $\ell^p_{a,b}$ exists: this is the space of all analytic functions $f(z)=\sum_{n\geq 0} \lambda_n f_n(z)$ on the disc $\mathbb{D}_R$, with $\{\lambda_n\}\in \ell^p(\mathbb{N}_0)$, where $1<p<\infty$. Similarly, we can define $\ell^1_{a,b}$ and $c_{0,a,b}$ by considering $\mathcal{D}_2$ and $\mathcal{D}_3$, respectively. Indeed, $c_{0,a,b}$ is the space of all analytic functions $f(z)=\sum_{n\geq 0} \lambda_n f_n(z)$ on the disc $\mathbb{D}_R$, where $\{\lambda_n\}\in c_0(\mathbb{N}_0)$, with the norm given by 
\[
\|f\|_{c_{0,a,b}}:=\sup_{n\geq 0}|\lambda_n|.
\]

A few comments are in order. It holds that every member $f(z)=\sum_{n=0}^{\infty}\lambda_n f_n(z)$ of $ \ell^p_{a,b}$ or $c_{0,a,b}$, admits a power series representation $f(z)=\lambda_0a_0+\sum_{n\geq 1} (\lambda_na_n+\lambda_{n-1}b_{n-1})z^n$, where $|z|<R$. Also, it follows from \eqref{series} that the evaluation functional 
\[
ev_z(f):=f(z)
\]
is continuous on $\ell^p_{a,b}$, at every point $z\in \mathcal{D}_1$, where $1<p<\infty$. In a similar way, $ev_z$ is continuous on $\ell^1_{a,b}$ and $c_{0,a,b}$, respectively when $z$ belongs to $\mathcal{D}_2$ or $\mathcal{D}_3$. The domain of $z\mapsto ev_z$ sometimes admits a suitable extension to points outside $\mathcal{D}_1$, and this makes the space $\ell^p_{a,b}$ very different from the usual weighted $\ell^p(\mathbb{N}_{0})$ spaces. Further, this extension property will be used to derive some unusual dynamical properties of $F_w^*$, cf. Section $3$. We would like to recall that such a phenomenon in the unilateral case when $p=2$, was also observed by Adams and McGuire in \cite{Adams-McGuire}. 

In what follows, the $n$-th coordinate functional (or rather, the $n$-th coefficient functional)  $k_n$ is given by 
\[
k_n(f)=\frac {{f}^{(n)}(0)}{n!},
\]
and it is defined on $\ell^p_{a,b}$ and $c_{0,a,b}$. Also, the coordinate functional with respect to the basis $\{f_n\}$ in $\ell^p_{a,b}$, $1\leq p<\infty$, or $c_{0,a,b}$, will be denoted by $f_n^*$, i.e.,
\begin{center}
$f_n^*(f)=\lambda_n,\hskip .6cm f=\sum_{n=0}^{\infty}\lambda_n f_n.$
\end{center}


\begin{proposition}\label{k_{n}} 
 The following hold \emph{(cf. \cite{Das-Mundayadan})}.
    \begin{itemize}
    \item[(i)] For $1<p<\infty$, $1/p+1/q=1$,
    \[
    \|k_n\|\leq (|a_n|^q+|b_{n-1}|^q)^{1/q}, n\geq 1.
    \]
    For $p=1$, we have $\|k_n\|\leq \max\{|a_n|,|b_{n-1}|\}$. Also, for the case of $c_{0,a,b}$, $\|k_n\|\leq |a_n|+|b_{n-1}|$. 
    \item[(ii)]  $\{f_n,f_n^*\}$ is a biorthogonal system, and for $1<p<\infty$, $\{f_n^*\}$ is an unconditional Schauder basis for the dual space.
    \item[(iii)] $k_0=a_0f_0^*$, and for $n\geq 1$, we have $k_n=a_nf_n^*+b_{n-1}f_{n-1}^*$.
   \end{itemize}
\end{proposition}

We will later require the fact that $\{k_n\}$ is a total set. We prove this in a more general set up of an analytic function space over an open disc.

\begin{proposition}\label{c-f}
    Let $\mathcal{E}$ be a Banach space of analytic functions on an open disc $B(0,r)$ centered around $0$, having evaluation functionals bounded at each $z\in B(0,r)$. Then, $k_n$ is bounded for all $n\geq 0$. Further, $\text{span}~ \big \{k_n:~n \geq 0\big \}$ is dense in $\mathcal{E}^*$ with respect to the weak$^*$-topology, and also, in the norm topology of $\mathcal{E}^*$ if it is reflexive.
\end{proposition}
\begin{proof}
    Let $ev_z$ denote the evaluation functional acting on $\mathcal{E}$. Since the  analyticity of $z\mapsto ev_z$ in the strong and norm operator topologies are equivalent (cf. Chapter 5, Theorem 1.2, \cite{Taylor}), we see that $z\mapsto ev_z$ is an $\mathcal{E}^*$-valued norm-analytic function. It then admits a norm-convergent power series:     $ev_z=\sum_{n=0}^{\infty} L_n z^n\hskip .5cm (z\in B(0,r)),$ for some $L_n\in \mathcal{E}^*$, $n\geq 0$. On the other hand, if $f\in \mathcal{E}$, we have $f(z)=ev_z(f)=\sum_{n=0}^{\infty} L_n(f)z^n.$ Expanding $f(z)$ in its Taylor series, we must then have $L_n=k_n$ for all $n$, and so each $k_{n}$ is bounded.

    To show that $\{k_n\}$ is a weak$^*$-total set in $\mathcal{E}^*$, let $\Lambda \in \mathcal{E}^{**}$ be such that $\Lambda(k_n)=0$ for all $n$. Hence, there is some $f\in \mathcal{E}$ such that 
    $\Lambda(k_n)=k_n(f)=0,$
 $n\geq 0$, which shows that all the coefficients of $f(z)$ vanish. Consequently, $f=0$.
    
\end{proof}

Let us find a condition under which the space $\ell^p_{a,b}$ contains all the polynomials. For this, note that we can express a monomial $z^n$ as an expansion in the basis $\{f_n\}$; see \eqref{monomial-expansion} below. Indeed, fix $n \geq 0,$ and write $z^{n}=\sum_{j=0}^{\infty}\alpha_{j}f_{j}$ for some $\alpha_{j} \in \mathbb{C}.$ Then $z^{n}=\alpha_{0}a_{0}+\sum_{j=1}^{\infty}(\alpha_{j-1}b_{j-1}+\alpha_{j}a_{j})z^{j}.$ Comparing the coefficients of the same powers, a simple computation gives that $\alpha_{0}=\alpha_{1}=\cdot \cdot \cdot=\alpha_{n-1}=0,$ $\alpha_{n}=\frac{1}{a_{n}}$, and $\alpha_{n+k}=\frac{(-1)^{k}}{a_{n}}\frac{b_{n}b_{n+1}\cdot\cdot \cdot b_{n+k-1}}{a_{n+1}a_{n+2}\cdot\cdot \cdot a_{n+k}}$, where $k\geq 1$. From these expressions, we obtain
\begin{equation}\label{monomial-expansion}
z^{n}=\frac{1}{a_{n}}\sum_{j=0}^{\infty}(-1)^{j}\left(\Mprod_{k=0}^{j-1}\frac{b_{n+k}}{a_{n+k+1}}\right)f_{n+j} \,\,\,\, (n\geq 0),
\end{equation}
where the coefficient for $j=0$ should be understood as $1$; see, also, p. $734$ of \cite{Adams-McGuire}, and p. 242 of \cite{Das-Mundayadan}. We have that $z^n \in \ell^p_{a,b}$ if and only if
\begin{equation} \label{m}
\|z^n\|_{\ell^p_{a,b}}^p:=\frac{1}{|a_{n}|^p}\sum_{j=0}^{\infty}\left|\Mprod_{k=0}^{j-1}\frac{b_{n+k}}{a_{n+k+1}}\right|^p<\infty.
\end{equation}
Once this has been satisfied, by linearity we obtain that any polynomial  belongs to $\ell^p_{a,b}.$

For a given weight sequence $w=(w_n)_{n=0}^{\infty}$, the operator given by 
\begin{center}
    $F_w(\sum_{n=0}^{\infty}\lambda_n z^n)=\sum_{n=0}^{\infty}\lambda_nw_nz^{n+1}$
    \end{center}
    on our spaces $\ell^p_{a,b}$ and $c_{0,a,b}$ is known as a \textit{weighted forward shift}. To derive some sufficient conditions for the unilateral shift $F_w$ to be bounded in terms of $\{a_n\}$, $\{b_n\}$ and the weight sequence $\{w_n\}$, we compute the matrix representation of the operator $F_w$ acting on $\ell^{p}_{a,b}$ and $c_{0,a,b}$ with respect to the ordered Schauder basis $\{f_n\}_{n\geq 0}$. For the special case of the unweighted forward shift $F$, some necessary and sufficient conditions for the continuity of $F$ on $\ell^2_{a,b}$ were obtained by Adams and McGuire \cite{Adams-McGuire}. Weighted backward shifts $B_w$ on $\ell^p_{a,b}$ were studied by the authors in \cite{Das-Mundayadan}. For operator theoretic results in the $p=2$ case, we refer to \cite{Das-Sarkar}.
 
Let us $F_w$ act on the normalized Schauder basis $f_{n}(z) =(a_{n}+b_{n}z)z^{n},n \geq 0$, of $\ell^{p}_{a,b}$. Then we obtain
\begin{eqnarray*}
F_{w}(f_{n})(z)&=&w_{n}a_{n}z^{n+1}+w_{n+1}b_{n}z^{n+2}=\frac{w_{n}a_n}{a_{n+1}}f_{{n+1}} + \left(\frac{w_{n+1}b_{n}}{a_{n+2}}-\frac{w_{n}a_n}{a_{n+1}}\frac{b_{n+1}}{a_{n+2}}\right)a_{n+2} z^{n+2}\\& &\hskip 1.7cm =\frac{w_{n}a_n}{a_{n+1}}f_{{n+1}}+c_{n}\sum_{j=0}^{\infty}(-1)^{j}\left(\Mprod_{k=0}^{j-1}\frac{b_{n+2+k}}{a_{n+3+k}}\right)f_{n+2+j},
\end{eqnarray*}
where $c_{n}:=w_{n+1}\frac{b_{n}}{a_{n+2}}-w_{n}\frac{a_n}{a_{n+1}}\frac{b_{n+1}}{a_{n+2}}$.

\noindent Therefore, the matrix of $F_{w}$ with respect to the basis $\{f_{n}\}_{n\geq 0}$ is
$$[F_{w}]:=\begin{bmatrix}
0 & 0 & 0&0&\cdots\\
w_{0}\frac{a_{0}}{a_{1}} & 0 & 0&0&\ddots\\
c_{0} & w_{1}\frac{a_{1}}{a_{2}} & 0&0&\ddots\\
-c_{0}\frac{b_{2}}{a_{3}} & c_{1} & w_{2}\frac{a_{2}}{a_{3}}&0&\ddots\\
 c_{0}\frac{b_{2}b_{3}}{a_{3}a_{4}}& -c_{1}\frac{b_{3}}{a_{4}} & c_{2}&w_{3}\frac{a_{3}}{a_{4}}&\ddots\\
-c_{0}\frac{b_{2}b_{3}b_{4}}{a_{3}a_{4}a_{5}} & c_{1}\frac{b_{3}b_{4}}{a_{4}a_{5}} &-c_{2}\frac{b_{4}}{a_{5}} &c_{3}&\ddots\\
\vdots &\ddots&\ddots&\ddots&\ddots
\end{bmatrix}.$$

\noindent We now provide a sufficient condition for $F_{w}$ to be bounded. For the particular cases of $p=2$ and the unweighted forward shift $F$, we refer to \cite{Adams-McGuire}. 

\begin{theorem}\label{decomposition}
   If 
   \[
 \sup_{n\geq 0}~\Big\{\left\lvert\frac{w_{n}a_{n}}{a_{n+1}}\right\rvert,~|c_n|\Big\}<\infty,~ \text{and}~
 \sum_{n=3}^{\infty}\sup_{j\geq 0} ~\left|c_j\Mprod_{k=3}^{n}\frac{ b_{j+k-1}}{a_{j+k}}\right|<\infty,
   \]
   then $F_{w}$ is a bounded operator on $\ell^{p}_{a,b}$ and $c_{0,a,b}$, where $1\leq p<\infty$.
   \end{theorem}
   \begin{proof}
The matrix of $F_w$ can be formally written as  $[F_{\alpha}]$+$\sum_{n=2}^{\infty} [F_{n}]$. Here, $[F_{\alpha}]$ is the matrix of the standard weighted forward shift $F_{\alpha}(e_i)\mapsto \alpha_i e_{i+1}$ on $\ell^p(\mathbb{N}_{0})$, $i\geq 1$, having weights 
\[
\alpha_i=\frac{ w_{i}a_i}{a_{i+1}} \hspace{1cm} (i\geq 0),
\]
on $\ell^p(\mathbb{N}_{0})$. Note that the matrix $[F_n]$ is obtained by deleting all the entries of $[F_{w}]$, except those at the $n$-th subdiagonal, where $n\geq 2$. Observe that $[F_n]$ is the matrix of suitable powers of a weighted forward shift $F_n$ for $n\geq 2$.

It follows, by the first assumption in the theorem, that the weighted shifts $F_{\alpha}$ and  $F_{2}$ are bounded on $\ell^p(\mathbb{N}_{0})$. Since
\[
\|F_n\|=\sup_{j\geq 0} ~\left|c_j\Mprod_{k=3}^{n}\frac{ b_{j+k-1}}{a_{j+k}}\right|,
\]
the second condition in the theorem gives that $F_n$ is bounded, and $\sum_{n\geq 3}\|F_n\|<\infty$, with respect to the operator norm. Hence, the shift $F_{w}$ is bounded. This completes the proof of the theorem.
\end{proof}
\begin{remark}
    The following conditions for the boundedness of $F_w$ are simpler and stronger, compared to those in the above theorem. If
\begin{equation}\label{sup-limsup}
\sup_{n\geq 0}~\left\lvert\frac{ w_{n}a_{n}}{a_{n+1}}\right\rvert<\infty \hspace{.4cm} \text{and} \hspace{.4cm} \lim \sup_{n}\left \lvert \frac{b_{n}}{a_{n+1}}\right \rvert<1,
\end{equation}
then the assumptions in Theorem \ref{decomposition} are easily satisfied, so that $F_w$ is bounded. See, also, \cite{Adams-McGuire}.
\end{remark}

We now obtain an important property of the adjoint $F_w^*$, which loosely says that the adjoint $F_w^*$ ``acts like" a weighted backward shift with respect to a total set, as proved below. The unweighted case of $F$ defined on reproducing kernel spaces over the unit disc was discussed in \cite{Mundayadan-Sarkar}.

\begin{proposition} \label{uni-shift}
Let $F_w$ be a bounded forward weighted shift on a Banach space $\mathcal{E}$ consisting of analytic functions on the unit disc $\mathbb{D}$, having continuous evaluation functionals. Then,
\begin{center}  
$F_w^*(k_n)=
\begin{cases}
    w_{n-1}k_{n-1}, \hskip.5cm n\geq 1,\\ 0,\hskip 2cm n=0.
\end{cases}$
\end{center}
\end{proposition}
\begin{proof}
 We have already showed that $k_n\in \mathcal{E}^*$ for all $n\geq 0$. For a function $f(z)=\sum_{k=0}^{\infty}\lambda_{k}z^{k}$ in $\mathcal{E}$, we have $k_{n-1}(f)=k_{n-1}\left( \sum_{k=0}^{\infty}\lambda_{k}z^{k}\right)=\lambda_{n-1}.$ On the other hand,
 \begin{eqnarray*}
  F_{w}^{*}\big(k_{n}\big)(f)&=&\big(k_{n}\circ F_{w}\big)(f)\\&=&k_{n}\big(\sum_{k=0}^{\infty} \lambda_{k}w_{k}z^{k+1}\big)=\lambda_{n-1}w_{n-1},
 \end{eqnarray*}
 which is the shifting property, as in the proposition.
\end{proof}

\section{Dynamics of the adjoint $F_w^*$}
In this section, we obtain a number of conditions for the hypercyclicity and chaoticity of the adjoint $F_w^*$. We, also, provide some curious orbital behavior of $F_w^*$. In the theorem below, we state the results for hypercyclicity only. For mixing of $F_w^*$, ``liminf" is replaced by ``lim", and we do not state it. Recall that an operator $T$ on a separable Banach space $X$ is \textit{supercyclic} if $\{\lambda T^nx:\lambda \in \mathbb{K}, n\geq 0\}$ is dense in $X$ for some vector $x\in X$, where $\mathbb{K}$ is $\mathbb{R}$ or $\mathbb{C}$. 

\begin{theorem}\label{unilateral}
     Let $\ell^{p}_{a,b},$ $1< p<\infty,$ be the analytic Banach space over the unit disc $\mathbb{D}$ corresponding to $a=\{a_n\}_{n=0}^{\infty}$ and $b=\{b_n\}_{n=0}^{\infty}$. Then, the following hold for the adjoint $F_{w}^{*}$, assuming that $F_w$ is bounded.
\begin{enumerate}
\item[(i)] If
\begin{equation}\label{hypercyclic3}
\liminf_{n\rightarrow \infty} ~ \frac{|a_{\nu+n}|+|b_{\nu+n-1}|} {\prod_{k=0}^{n-1}|w_{\nu+k}|}=0,~\forall~\nu\geq 0,
\end{equation}
then $F_w^*$ is hypercyclic.
\item[(ii)] Conversely, if $F_w^*$ is hypercyclic, and $1\in \ell^p_{a,b}$ then
\[
\sup_{n\geq 0} ~\frac{\prod_{k=0}^{n-1}|w_k|}{|a_{n}|} \left (1+\sum_{j=1}^{\infty}\left|\Mprod_{k=0}^{j-1}\frac{b_{n+k}}{a_{n+k+1}}\right|^p\right)^{1/p}=\infty.
\]

\item[(iii)] The adjoint operator $F_w^*$ is always supercyclic.
\item[(iv)] If the conditions in \eqref{sup-limsup} are satisfied, then $F_{w}^{*}$ is hypercyclic on $(\ell^{p}_{a,b})^{*}$ if and only if 
         \begin{equation}\label{iff hyper}
               \liminf_{n\rightarrow \infty} \left| \frac{a_{\nu+n }}{\prod_{k=0}^{n-1} w_{\nu+k}}\right|=0,~~ \forall~~ \nu\in \mathbb{N}.
        \end{equation}
         \end{enumerate}
\end{theorem}
\begin{proof}  (i) First suppose that (\ref{hypercyclic3}) holds, and we will use Gethner-Shapiro criterion to show that $F_{w}^{*}$ is hypercyclic. Consider $D:=\text{span}\{k_{\nu}:\nu \geq 0\}$. Then $D$ is dense in $(\ell^{p}_{a,b})^{*},$ by Proposition \ref{c-f}. Define $S:D\rightarrow D$ by 
\[
S(k_{\nu})=\frac{1}{w_{\nu}}k_{\nu+1},
\]
$\forall~~ \nu\geq 0.$ Since $F_{w}^{*}S(f)=f$ and $(F_{w}^{*})^{n}f\to 0,$ as $n \to \infty,$  $\forall$ $f\in D,$ we just need to show that there exists $(n_k)$ such that $S^{n_{k}}(k_{\nu})$ tends to $0$, as $k \to \infty,$ for $\nu\geq 0$. Indeed, Proposition \ref{k_{n}} along with our assumption in (i) gives that, for some $(n_k)$,
       \[
       \lVert S^{n_{k}}(k_{\nu}) \rVert=\left\lVert\frac{k_{\nu+n_{k} }}{w_{\nu}w_{\nu+1}\cdots w_{\nu+n_{k}-1}} \right\rVert\leq M_{1}\frac{|a_{\nu+n_{k}|}+|b_{\nu+n_k-1}|}{|w_{\nu}w_{\nu+1}\cdots w_{\nu+n_{k}-1}|} \longrightarrow 0,
       \]
       as $k\rightarrow \infty$, where $M_{1}$ is a constant. Thus, $ F_{w}^{*}$ satisfies the Gethner-Shapiro criterion. 
       
       (ii) Conversely, if  $F_{w}^{*}$ is hypercyclic on $(\ell^{p}_{a,b})^{*}$, then by a result of Bonet \cite{Bonet1}, we have 
       \[
       \sup_{n\geq 1} \|w_0w_1\cdots w_{n-1}z^n\|=\sup_{n\geq 1} \|F_w^n(1)\|=\infty.
       \]
       By the norm expression \eqref{m}, we get the result in (ii).

       (iii) For proving the supercyclicity of $F_w^*$, we recall a form of the supercyclicity criterion: For a bounded operator $T$ defined on a separable Banach space $X$, if there exist a dense set $D$ in $X$ and a map $S:D\rightarrow D$ such that $TSx=x$ and $\|T^nx\|\|S^ny\|\rightarrow 0$, as $n\rightarrow \infty$, for all $x,y\in D$, then $T$ is supercyclic. We refer to \cite{Bayart-Matheron}, p. 9, for a general form of this criterion.

       Now, apply the supercyclicity criterion with $D$ and $S$ as in (i), taking $n_k=k$, for $k\geq 1$. By Proposition \ref{uni-shift}, for each $\nu\geq 0$ the sequence $F_w^{*n}(k_{\nu})$ vanishes for large $n$, and hence $F_w^*$ satisfies the supercyclicity criterion.

       (iv) By the assumptions in \eqref{sup-limsup}, $F_w$ becomes a bounded operator. Also, one can find $r<1$ and $N\in \mathbb{N}$ such that $\big| b_n/a_{n+1}\big|<r$ for every $n\geq N$. Hence, the conditions that appear in (i) and (ii) are easily seen to be equivalent to \eqref{iff hyper}. 

\end{proof}

 For a special case, we refer to \cite{Mundayadan-Sarkar} where the authors provided some results on the dynamics of $F^*$ in tridiagonal spaces. Our next result deals with the chaos of $F_w^*$. In order to apply the chaoticity criterion, we need a representation of $f_n^*$ in terms of $k_n$. 

\begin{lemma}\label{rep-unilateral}
     Let $k_n$ denote the coordinate functional $f\mapsto \frac{f^{(n)}(0)}{n!},$ defined on $\ell^p_{a,b}$ or $c_{0,a,b}$, where $n\geq 0$. Then we have
      \[
        f_{n}^{*}= \frac{1}{a_{n}}k_{n}+\sum_{j=1}^{n}(-1)^{j}\frac{b_{n-1}b_{n-2}\cdots b_{n-j}}{a_{n}a_{n-1}\cdots a_{n-j}}k_{n-j},
       \]
       for all $n\geq1$ and $f_{0}^{*}=\frac{1}{a_{0}}k_{0}.$
\end{lemma}

\begin{proof}
   If $f \in \ell^{p}_{a,b},$ then $f(z)=\sum_{n=0}^{\infty} \lambda_{n}f_{n}(z)=\sum_{n=0}^{\infty}\lambda_{n}(a_{n}+b_{n}z)z^{n}.$ Arranging into a power series, we note that $f(z)=\lambda_0a_0+\sum_{n\geq 1} (\lambda_na_n+\lambda_{n-1}b_{n-1})z^n,$ 
 and so, $k_n(f)=\lambda_na_n+\lambda_{n-1}b_{n-1}$ and $k_{0}(f)=\lambda_{0}a_{0}.$ Thus, we have
 \begin{eqnarray*}
     \frac{1}{a_{n}}k_{n}(f)-\cdots +(-1)^{n}\frac{b_{n-1}\cdots b_{0}}{a_{n}\cdots a_{0}}k_{0}(f)&=&\frac{1}{a_{n}}(\lambda_na_n+\lambda_{n-1}b_{n-1})-\cdots+(-1)^{n}\frac{b_{n-1}\cdots b_{0}}{a_{n}\cdots a_{0}}(\lambda_{0}a_{0})\\
    &=& \lambda_{n}\\
     &=&  f_{n}^{*}(f),
 \end{eqnarray*}
 for all $n\geq 1$. Moreover, $\frac{1}{a_{0}}k_{0}(f)=\frac{1}{a_{0}}(\lambda_{0}a_{0})=\lambda_{0}=f_{0}^{*}(f).$
 \end{proof}

We now obtain necessary and sufficient conditions for $F_w^*$ to be chaotic on the dual of $\ell^p_{a,b}$, stated below. We recall that in \cite{Mundayadan-Sarkar} the authors provided a sufficient condition for the adjoint $F^*$ of the unweighted shift $F$ to be chaotic on a general analytic reproducing kernel space over $\mathbb{D}$ which includes the case of $\ell^2_{a,b}$. 

\begin{theorem}\label{unilateral1}
     Assume that the conditions stated in \eqref{sup-limsup} are satisfied for $\ell^p_{a,b}$, where $1<p<\infty$. Consider the following statements:
     \begin{enumerate}
\item[(i)]$F_{w}^{*}$ is chaotic on $(\ell^{p}_{a,b})^{*}.$
\item[(ii)] $F_{w}^{*}$   has a non-trivial periodic vector.
\item[(iii)] $\sum_{n=1}^{\infty}\left\lvert \frac{a_{n}}{w_{0}w_{1}\cdots w_{n-1}} \right\rvert^{q} < \infty.$ 
\end{enumerate}

Then $(iii)\Rightarrow (i)\Rightarrow (ii)$. In addition, if $\frac{1}{p}+\frac{1}{q}=1$ and
\[
       \sum_{n\geq0}\left(\sum_{j\geq 1}\left\lvert \Mprod_{k=1}^{j}\frac{b_{n+k-1}}{ a_{n+k}}\right\rvert^{p}\right)^{\frac{q}{p}}<\infty,
      \]
      then $(ii)\Rightarrow (iii)$.
\end{theorem}
\begin{proof}
    Clearly, \text{(i)} gives \text{(ii)}.
    
   To see that \text{(ii)} implies \text{(iii)}, let $\psi=\sum_{m=0}^{\infty}\lambda_{m}f_{m}^{*}$ be a non-zero periodic vector for $F_{w}^{*}$ on $(\ell^{p}_{a,b})^{*}$, where $f_m^{*}$, $m\geq 0$, form an unconditional Schauder basis for $(\ell^{p}_{a,b})^{*}$. Now, using Lemma \ref{rep-unilateral}, we get
       \[
       \psi=\sum_{m=0}^{\infty} A_{m}k_{m},~~~~\text{where}~~~A_{m}=\lambda_{m}\frac{1}{a_{m}}+\sum_{j=1}^{\infty}(-1)^{j}\lambda_{m+j}\frac{b_{m}b_{m+1}\cdots b_{m+j-1}}{a_{m}a_{m+1}\cdots a_{m+j}},
       \]
       for all $m\geq 0.$ Now, let $ \nu \in \mathbb{N}$ be such that $(F_{w}^{*})^{\nu}\psi=\psi$. Then $(F_{w}^{*})^{n\nu}\psi=\psi$ for all $n\geq 1$, which imply that
\begin{center}
    $\displaystyle{\sum_{m=n\nu}^{\infty}}w_{m-1}w_{m-2}\cdots w_{m-n\nu }A_{m}k_{m-n\nu}=\sum_{m=0}^{\infty}A_{m}k_{m}. $
\end{center}
       Equating the coefficient of $k_{0}$ both sides, we get
       \[
       \frac{w_{n\nu-1} w_{n\nu-2}\cdots w_{0}}{a_{n\nu}}\left( \lambda_{n\nu} +\sum_{j=1}^{\infty}(-1)^{j}\lambda_{n\nu+j}\frac{b_{n\nu}\cdots b_{n\nu+j-1}}{a_{n\nu+1}\cdots a_{n\nu+j}} \right)= \frac{1}{a_{0}}\left( \lambda_{0}+\sum_{j=1}^{\infty}(-1)^{j}\lambda_{j}\frac{b_{0}\cdots b_{j-1}}{a_{1}\cdots a_{j}}\right).
       \]
       Since the series in the right hand side converges, we get that
       \[
       C \sum_{n=1}^{\infty} \left\lvert \frac{a_{n\nu}}{w_{n\nu-1} w_{n\nu-2}\cdots w_{0}} \right\rvert^{q}= \left\lvert a_{0}  \right\rvert^{q}\sum_{n=1}^{\infty} \left\lvert \lambda_{n\nu} +\sum_{j=1}^{\infty}(-1)^{j}\lambda_{n\nu+j}\frac{b_{n\nu}\cdots b_{n\nu+j-1}}{a_{n\nu+1}\cdots a_{n\nu+j}} \right\rvert^{q},
       \]
       for some $C>0$. Note that using the H\"{o}lder inequality, we can get
       \[
       \sum_{n=1}^{\infty} \left\lvert\lambda_{n\nu} +\sum_{j=1}^{\infty}(-1)^{j}\lambda_{n\nu+j}\frac{b_{n\nu}\cdots b_{n\nu+j-1}}{a_{n\nu+1}\cdots a_{n\nu+j}}\right\rvert^{q}<\infty,
       \]
        as $\{\lambda_{n}\}_{n}\in \ell^{q}(\mathbb{N}_{0})$ and  $ \sum_{n\geq0}\left(\sum_{j\geq 1}\left\lvert \displaystyle\Mprod_{k=1}^{j}\frac{ b_{n+k-1}}{a_{n+k}}\right\rvert^{p}\right)^{\frac{q}{p}}<\infty.$ Therefore,\\ $\sum_{n=1}^{\infty}\left\lvert \frac{a_{n\nu}}{w_{n\nu-1} w_{n\nu-2}\cdots w_{0}} \right\rvert^{q}<\infty.$ Similarly, equating the coefficient of $k_{j}, j=1,2,\cdots,\nu-1$, and performing the calculations as above, we obtain that $\sum_{n=1}^{\infty}\left\lvert \frac {a_{n\nu+j}}{w_{n\nu+j-1} w_{n\nu+j-2}\cdots w_{j}}\right\rvert^{q}<\infty.$ We conclude that $\sum_{n=1}^{\infty}\left\lvert \frac{a_{n}}{w_{0}w_{1}\cdots w_{n-1}}\right\rvert^{q}<\infty.$

      Now, we prove that \text{(iii)} implies \text{(i)}, by applying the chaoticity criterion to show that $F_{w}^{*}$ is chaotic on $(\ell^{p}_{a,b})^{*}.$ Consider $D=\text{span}\{k_{n}:n \geq 0\}$ which is dense in $(\ell^{p}_{a,b})^{*}.$  Define $S:D\rightarrow D$ by $S(k_{n})=\frac{1}{w_{n}}k_{n+1},~~ \forall~~ n\geq 0.$ Clearly, $F_{w}^{*}S(f)=f,$ and the series $\sum_{n=1}^{\infty}(F_{w}^{*})^{n}(f)$ converges unconditionally for each $f \in D.$  We just need to show that $\sum_{n=1}^{\infty}S^{n}(f)$ converges unconditionally, for each $f\in D.$ By the well known Orlicz-Pettis theorem (cf. Diestel \cite{Diestel}) it suffices to see that the series is weakly unconditionally convergent, that is, 
      \begin{equation}\label{wuc}
      \sum_{n\geq 1}\left|\frac{k_n(f)}{w_{0}w_{1}\cdots w_{n-1}}\right|<\infty,~~f\in \ell^p_{a,b}.
      \end{equation}
       This can be verified as follows: we have $k_n(f)=\lambda_n a_n+\lambda_{n-1}b_{n-1}$, for $n\geq 1$ and $f=\sum_{n=0}^{\infty}\lambda_nf_n(z)\in \ell^p_{a,b}$.
As $\limsup_n |b_n/a_{n+1}|<1$, one gets $N\in\mathbb{N}$ and $r<1$ such that $|b_n/a_{n+1}|<r$ for all $n\geq N$. It follows that
\[
\left\lvert \frac{\lambda_na_{n}}{w_{0}w_{1}\cdots w_{n-1}}+\frac{\lambda_{n-1}a_{n}}{w_{0}w_{1}\cdots w_{n-1}}\frac{b_{n-1}}{a_{n}}  \right\rvert\leq \left\lvert\frac{a_{n}\lambda_n}{w_{0}w_{1}\cdots w_{n-1}} \right\rvert+r\left\lvert\frac{a_{n}\lambda_{n-1}}{w_{0}w_{1}\cdots w_{n-1}} \right\rvert,
\]
for $n\geq N$. Since $\left\{\frac {a_n} {w_0\cdots w_{n-1}}\right\}\in \ell^q(\mathbb{N}_{0})$ and $\{\lambda_n\}\in \ell^p(\mathbb{N}_{0})$, the H\"{o}lder inequality yields that the condition in \eqref{wuc} is satisfied. Our conclusion now, is that $F_{w}^{*}$ satisfies the chaoticity criterion, and \text{(i)} follows.
\end{proof}

The proofs of the hypercyclicity and chaos for $F_w^*$ show that, in terms of coefficient functionals, one has more general results on the dynamics of $F_w^*$ when $F_w$ is defined on a Banach space of analytic functions, stated as follows. Its proof is omitted.

\begin{theorem}\label{gen}
    Let $\mathcal{E}$ be a Banach space of analytic functions on the unit disc $\mathbb{D}$, having norm-separable dual $\mathcal{E}^*$. Assume that $\{k_n: n \geq 0\}$ is linearly independent, and that the operator $F_w$ is bounded on $\mathcal{E}$. Then the following  properties hold.
    \begin{itemize}
        \item[(i)] If $\inf_{n\geq 0} |\prod_{k=0}^{n-1} w_{k}|^{-1}\|k_{n}\|=0$, then $F_w^*$ is hypercyclic.
        \item[(ii)] If $\lim_{n\rightarrow \infty} |\prod_{k=0}^{n-1} w_{k}|^{-1}\|k_{n}\|=0$, then $F_w^*$ is mixing.
        \item[(iii)] If $\mathcal{E}$ is reflexive, and 
        \[
        \sum_{n\geq 1}\left|\left (\Mprod_{k=0}^{n-1} w_{k}\right)^{-1} \frac{f^{(n)}(0)}{{(n)}!}\right|<\infty,
        \]
        for all $f\in \mathcal{E}$, then $F_w^*$ is chaotic.
    \end{itemize}
\end{theorem}





\begin{remark}
Recall that the dynamics of a weighted backward shift $B_w$ on $\ell^p_{a,b}$ was studied in \cite{Das-Mundayadan}. The chaos part was established via the unconditional convergence of the series $\sum_{n=1}^{\infty}\frac{1}{w_1\cdots w_n}z^n$ (see Theorem $4.7$ there). We would like to provide, here, a correction to the argument involving the unconditional convergence of the above series along with additional assumptions. For this we need an identification of the dual $\ell^p_{a,b}$ with $\ell^q(\mathbb{N}_0)$, where $1\leq p<\infty$, and $1/p+1/q=1$. We consider the case $1<p<\infty$ only, and note that each $y=(y_j)_{j=0}^{\infty}\in \ell^{q}(\mathbb{N}_0)$ defines a bounded functional $L_y$ on $\ell^p_{a,b}$ in the following way: $L_y(\sum_{j\geq 0} \lambda_j f_j(z)):=\sum_{j\geq 0}\lambda_j y_j$. Conversely, if $L\in (\ell^p_{a,b})^*$, then there exists a unique $y\in \ell^q(\mathbb{N}_0)$ such that $L=L_y$. To see the later, let $T:\ell^p(\mathbb{N}_0)\rightarrow \ell^p_{a,b}$ be the linear isomorphism given by $T(\lambda_0,\lambda_1,\cdots)=\sum_{j\geq 0} \lambda_j f_j$. The composition $LT$ defines a continuous functional on $\ell^p(\mathbb{N}_0)$, and one gets a unique 
$y\in \ell^q(\mathbb{N}_0)$ such that $LT(\lambda_0,\lambda_1,\cdots)=\sum_{j=0}^{\infty} \lambda_j y_j$, whence $L=L_y$.

We now prove the following. For $1<p<\infty$, if $\ell^p_{a,b}$ contains all the polynomials, 
\[
\sum_{n=1}^{\infty}\frac{1}{|w_{1}\cdots w_{n}a_{n}|^{p}}<\infty\hspace{1cm}\text{and}\hspace{1cm} \sum_{n\geq 1}\left(\sum_{j\geq 0}\left\lvert \Mprod_{k=1}^{j}\frac{b_{n+k-1}}{ a_{n+k}}\right\rvert^{p}\right)^{\frac{q}{p}}<\infty,
\]
then $B_w$ is chaotic. It suffices to show that $\sum_{n=1}^{\infty}\frac{1}{w_{1}\cdots w_{n}}z^{n}$ is weakly unconditionally convergent. If $L\in (\ell^{p}_{a,b})^{*},$ then there exists $\{y_{n}\}_{n=0}^{\infty}\in \ell^{q}(\mathbb{N}_{0})$ such that $L\left(\sum_{n=0}^{\infty}\lambda_{n}f_{n}(z)\right)=\sum_{n=0}^{\infty}\lambda_{n}y_{n}.$ So,
\[
L(z^{n})=\frac{1}{a_{n}}\sum_{j=0}^{\infty}\lambda_{n,j}L(f_{n+j})=\frac{1}{a_{n}}\sum_{j=0}^{\infty}\lambda_{n,j}y_{n+j}.
\]
Hence,
\begin{eqnarray*}
    \sum_{n=1}^{\infty}\left\lvert\frac{1}{w_1\cdots w_n}L(z^n)\right\rvert&\leq& \sum_{n=1}^{\infty}\frac{1}{|w_{1}\cdots w_{n}a_{n}|}\left\lvert\sum_{j=0}^{\infty}\lambda_{n,j}y_{n+j}\right\rvert\\
    &\leq&\left(\sum_{n=1}^{\infty}\frac{1}{|w_{1}\cdots w_{n}a_{n}|^{p}}\right)^{\frac{1}{p}}\left(\sum_{n=1}^{\infty}\left(\sum_{j=0}^{\infty}|\lambda_{n,j}|^p\right)^{\frac{q}{p}}\right)^{\frac{1}{q}},
\end{eqnarray*}
which gives the desired result. By a suitable modification of the preceding argument, a similar result can be obtained for $p=1$.

\end{remark}

     We close this section with results on compact perturbations, essential spectrum, and hypercyclic subspaces for $F_w^*$ when $F_w$ is defined on $\ell^p_{a,b}$. Recall that the essential spectrum of a unilateral shift on $\ell^2(\mathbb{N}_{0})$ was computed in Shields \cite{Shields}. The same proof works for $\ell^p(\mathbb{N}_{0})$, $1<p<\infty$.
     
     \begin{proposition}\label{perturbation} 
         Suppose that the weighted shift $F_w$ is bounded on $\ell^p_{a,b}$, where $1<p<\infty$. If $
c_{n}:=w_{n+1}\frac{b_{n}}{a_{n+2}}-w_{n}\frac{a_n}{a_{n+1}}\frac{b_{n+1}}{a_{n+2}}\rightarrow 0$, as $n\rightarrow \infty$, then $F_w$ is similar to a perturbation $F_{\alpha}+K$ for some compact operator $K$ and a weighted forward shift $F_{\alpha}$ defined on $\ell^p(\mathbb{N}_{0})$, having weights $\alpha_i= \frac{w_ia_i}{a_{i+1}}$, and hence $\sigma_e(F_w)=\sigma_e(F_{\alpha})$.
     \end{proposition}
     \begin{proof}
         Recall from the matrix representation of $F_w$ on $\ell^p_{a,b}$, it follows that $F_w$ is similar to an operator of the form $K+F_{\alpha}$ for some bounded operator $K$ and a weighted forward shift $F_{\alpha}$ on $\ell^p(\mathbb{N}_{0})$. Here, $K=\sum_{n\geq 2} F_n$, where $F_n$ is as in the proof of Theorem \ref{decomposition}. Since $c_n$ is a null sequence, each $F_n$ becomes compact, and so $K$ is compact. As the essential spectrum is invariant under compact perturbations, we also have $\sigma_e(F_w)=\sigma_e(F_{\alpha})$.
     \end{proof}
     \begin{corollary}
Suppose that $F_w^*$ is hypercyclic on the dual $(\ell^p_{a,b})^*$, where $1<p<\infty$, and $1/p+1/q=1$. Under the assumptions as in the previous proposition, the adjoint $F_w^*$ possesses a hypercyclic subspace if and only if $F_{\alpha}^*$ has a hypercyclic subspace in $\ell^q(\mathbb{N}_{0})$.
     \end{corollary}
     We illustrate the above results on the unilateral weighted shift $F_w$ with an example. We actually get a class of chaotic operators that are compact perturbations of weighted backward shifts on $\ell^p(\mathbb{N}_0)$.
     
 \begin{example}
    Choose $a_{0}=b_{0}=w_{0}=1,$ and $a_{n}=\frac{1}{n2^{n-1}},~ b_{n}=\frac{1}{(n+1)4^{n}},~ w_{n}=4, ~n\geq 1$. Then
    \[
\frac{1}{R}=\limsup_{n\rightarrow \infty} ~(|a_n|+|b_n|)^{1/n}=\limsup_{n\rightarrow \infty}\left(\frac{1}{n2^{n-1}}+\frac{1}{(n+1)2^{2n}}\right)^{1/n}=\frac{1}{2},
\]
 and $F_{w}$ is bounded on $\ell^p_{a,b},$ as it satisfies the conditions of Theorem \ref{decomposition}. Also, it should be noted that for the given values of $\{a_{n}\}_{n\geq 0}$ and $\{b_{n}\}_{n\geq 0},$ we have 
    \[
    w_{n+1}\frac{b_{n}}{a_{n+2}}-w_{n}\frac{a_n}{a_{n+1}}\frac{b_{n+1}}{a_{n+2}}=\frac{4}{2^{n}}\frac{n^{2}+2n-1}{n(n+1)}\to 0~~~~~~\text{as}~~~~~n\to \infty.
    \] 
Therefore, according to Proposition \ref{perturbation}, the operator $F_{w}$ on $\ell^{p}_{a,b}$ is similar to a compact perturbation of a weighted forward shift on $\ell^p(\mathbb{N}_{0}).$ Again, by the Theorem \ref{unilateral1}, the adjoint $F_{w}^{*}$ is  mixing and chaotic on $(\ell^p_{a,b})^{*},$ $1<p<\infty$.
\end{example}

\subsection{Notable differences: eigenvectors, periodic vectors, and a zero-one law}
We establish a couple of curious results which show that the forward weighted shifts $F_{w}$ on the space $\ell^p_{a,b}$ behave very differently from the classical unilateral weighted shifts. The adjoint of an unweighted forward shift operator $F$ can have a non-zero periodic vector without $F^{*}$ being hypercyclic, which is in sharp contrast to the case of the classical weighted shifts.  On another direction, a well known result states that if a unilateral or bilateral weighted backward shift on the corresponding $\ell^p$ spaces, has an orbit admitting a non-zero limit point, then the operator has a dense orbit, cf. Chan and Seceleanu \cite{Kit1}. This property is known as a zero-one law of orbital limit points for hypercyclicity, and we refer to Abakumov and Abbar \cite{Abakumov-Abbar},  Bonilla et al. \cite{Bonilla-Cardeccia}, and Chan and Seceleanu \cite{Kit} for more on this topic. Such a law does not hold for $F^*_w$ defined on the dual of $\ell^p_{a,b}$, as proved below. These interesting properties of $F^*$ will be based on a suitable extension of the domain of the map $z\mapsto ev_z$, where $ev_z$ denotes the evaluation functional at $z$. Recall that in \cite{Bonilla-Cardeccia} the authors have obtained a non-hypercyclic weighted backward shift on $c_0(\mathbb{Z})$, with an orbit admitting non-zero weak sequential limit points. (We say that a vector $u$ is a weak sequential limit point of a subset $A$ of a Banach space $X$ if there exists $\{u_n\}$ in $A$, which converges to $u$ weakly.)

We require the following necessary condition for non-hypercyclicity.

\begin{proposition}
Let $F$ be bounded on $\ell^p_{a,b}$, where $a_n=1$ for all $n\geq 0$, and let $1\in \ell^p_{a,b} $. If there exists $\mu \in \mathbb{C}$ such that
\[
\sup_{n\geq 0}\left(|\mu+b_{n}|^{p}\Big(1+\sum_{j=1}^{\infty}
\Big|\prod_{k=1}^{j} b_{n+k}\Big|^{p}\Big)\right)<\infty,
\]
then $F^*$ is not hypercyclic on $(\ell^p_{a,b})^{*}$.
\end{proposition}
\begin{proof}
     Note that, since $1\in \ell^p_{a,b}$, and $F$ is well defined, the space $\ell^p_{a,b}$ contains all polynomials by linearity. We show that $F$ has a bounded (non-zero) orbit. For this, let $f(z)$ be the function $1-\mu z$. Now, using the expansion of monomials with respect to the basis $\{f_n\}$ 
given in ~\eqref{monomial-expansion}, we obtain
\[
F^{n}(f)(z)=z^{n}-\mu z^{n+1}
= f_{n}-(\mu+b_{n})\left( f_{n+1}
+\sum_{j=1}^{\infty}(-1)^j\Big(\prod_{k=1}^{j} b_{n+k}\Big)
f_{n+j+1}\right).
\]
Consequently,
\[
\| F^{n}(f) \|^{p}_{\ell^{p}_{a,b}}
= 1+|\mu+b_{n}|^{p}\left(1 +\sum_{j=1}^{\infty}
\Big|\prod_{k=1}^{j} b_{n+k}\Big|^{p}\right).
\]
By our assumption, we have
\[
\sup_{n\geq 0} \, \| F^{n}(f) \|^{p}_{\ell^{p}_{a,b}} < \infty,
\]
from which it follows that $F^{*}$ is non-hypercyclic on $(\ell^{p}_{a,b})^{*}$.
\end{proof}
The space we consider below is based on Theorem $7$ in Adams and McGuire \cite{Adams-McGuire}, p. $734$, where the authors studied the unweighted shift on a tridiagonal kernel space.

\begin{theorem} \label{pr}
Let $|\lambda|=1$, $1<p<\infty$, and $1/p+1/q=1$. Let the forward shift $F$ be bounded on $\ell^p_{a,b}$, where
\begin{enumerate}
    \item[\textnormal{(A1)}] $a_{n}=1,~~~ 0<|b_n|<1,~~ n\geq 0~~~,\text{and}~~~\sum_{n\geq 0} |1+\lambda b_n|^q<\infty,$
\item[\textnormal{(A2)}]  $\sup_{n\geq 0}\left(|1+\lambda b_{n}|^{p}\left(\sum_{j=1}^{\infty}
\left|\prod_{k=1}^{j} b_{n+k}
\right|^{p}\right)\right)<\infty$,
\item[\textnormal{(A3)}] the radial limit $\lim_{z\rightarrow \lambda} f(z)$ exists and equals to $\sum_{n\geq 0}\lambda_n (1+b_n\lambda)\lambda^n$, for every function $f(z)=\sum_{n\geq 0}\lambda_n (1+b_nz)z^n$ in $\ell^p_{a,b}.$
\end{enumerate}
Then, the following hold.
\begin{enumerate}
\item[\textnormal{(i)}] $F^*$ is not hypercyclic.
\item[\textnormal{(ii)}]  The zero-one law of hypercyclicity fails for $F^*$.
\item[\textnormal{(iii)}] If $\lambda$ is a root of the unity, then $F^{*}$ has a non-trivial periodic vector in $(\ell^p_{a,b})^{*}$.
\end{enumerate}
\end{theorem}
\begin{proof}
Note that according to the assumption (A2), 
\[
\|z^{n}\|_{\ell^p_{a,b}}^{p}
=1+|b_{n}|^{p}\left(1 +\sum_{j=1}^{\infty}
\Big|\prod_{k=1}^{j} b_{n+k}\Big|^{p}\right)
<\infty.
\]
Hence, by~\eqref{m}, it follows that $z^{n}\in \ell^{p}_{a,b}$, for all $n\geq 0$. 

Using the previous proposition, for $\mu=\frac{1}{\lambda},$ we note that
\[
\sup_{n\geq 0}\left(|1+\lambda b_{n}|^{p}\Big(1+\sum_{j=1}^{\infty}
\Big|\prod_{k=1}^{j} b_{n+k}
\Big|^{p}\Big)\right)<\infty,
\]
because of the assumption $\textnormal{(A2)}.$ Therefore, $F^{*}$ is not hypercyclic on $(\ell^{p}_{a,b})^{*}$, and hence, (i) follows.\\
The  forward shift operator $F$ in $\ell^{p}_{a,b}$  is given by $F(f(z))=zf(z),$ for all $f \in \ell^{p}_{a,b}$ and $z \in \mathbb{D}.$ Now, since $F$ is bounded on $\ell^{p}_{a,b},$ for $f \in \ell^{p}_{a,b}$ and $z \in \mathbb{D},$ it can be checked that $F^{*}(ev_{z})=z ev_z$, i.e., $z$ is an eigenvalue of $F^{*}$ corresponding to the eigenvector $ev_{z}$ on $(\ell^{p}_{a,b})^{*},$ provided $ev_z$ is a non-zero functional. This fact is well known, cf. Godefroy and Shapiro \cite{Godefroy-Shapiro}. 

In our case, $a_{n}=1,$ $0<\lvert b_{n} \rvert<1,$ $n\geq 0,$ and $\lambda$ is unimodular. Now,
\begin{equation}\label{eevv}
\sum_{n=0}^{\infty} \lvert f_{n}(\lambda)\rvert^{q}= \sum_{n=0}^{\infty} |1+\lambda b_n|^q <\infty,
\end{equation}
for $1<q<\infty.$ Recall that if $f\in \ell^p_{a,b}$, then 
\begin{center}
$f(z)=\sum_{n\geq 0} \widehat{f}(n) f_n(z)$, and $\sum_{n\geq 0}|\widehat{f}(n)|^p<\infty$.
\end{center}
In view of \eqref{eevv}, we obtain that each $f(z)$ in $\ell^p_{a,b}$ extends to $z=\lambda$, and thus, it is meaningful to define 
\[
ev_{\lambda}(f)=f(\lambda):=\sum_{n\geq 0} \widehat{f}(n) f_n(\lambda),
\]
for all $f$. Since $|f(\lambda)|\leq M_{\lambda}\|f\|$ for some constant $M_{\lambda}>0$, we have that $ev_{\lambda} \in (\ell^{p}_{a,b})^{*}$. Note that $ev_{\lambda}$ is non-zero as, for example, $ev_{\lambda}(g)\neq 0$, where 
\[
g(z)=\sum_{n\geq 0} \frac{1}{\lambda^n 2^n} \frac{|1+\lambda b_n|}{1+\lambda b_n}f_n(z) \in \ell^p_{a,b}.
\]
 \noindent We claim that $F^{*}(ev_{\lambda})=\lambda(ev_{\lambda})$. To prove this claim, by (A3), we have that if $f\in \ell^p_{a,b},$ then the radial extension of $f(z)$ to $f(\lambda)$ is continuous, and thus $zf(z)\rightarrow \lambda f(\lambda)$ as $z\rightarrow \lambda$ radially. Hence, we have $zev_z\rightarrow \lambda ev_{\lambda}$ radially, pointwise. On the other hand, since $ev_z\rightarrow ev_{\lambda}$ as $z\rightarrow \lambda$ radially, we have the (radial) pointwise limit $F^*{ev_z}\rightarrow F^*(ev_{\lambda})$. Since $F^*{ev_z}=zev_z$ for $|z|<1$, the claim is proved.

If $\lambda$ is a root of unity, say $\lambda^m=1$, then $F^{*m}(ev_{\lambda})=ev_{\lambda}$. Hence, $F^{*}$ has a nontrivial periodic vector, namely $ev_{\lambda}$, which gives (iii).

To prove (ii), consider the orbit of $f:=ev_z+ev_{\lambda}$ under $F^*$, for any $z\in \mathbb{D}$. If $\lambda^j=1$, then $(F^*)^{jn}(f)\rightarrow ev_{\lambda}$, as $n\rightarrow \infty$. If $\lambda$ is not a root of unity, then we use the well known fact, due to Kronecker's density theorem, that for such a $\lambda$ the sequence $
\{\lambda^n:n\geq 1\}$ is dense in the unit circle of $\mathbb{C}$. In particular, we obtain a sequence $(j_n)$ such that $\lambda^{j_n}\rightarrow 1$, as $n\rightarrow \infty$. Consequently, we have 
\begin{center}
$(F^*)^{j_n}(f)\rightarrow ev_{\lambda}$,
\end{center}
as $n\rightarrow \infty$.
Thus, the zero-one law of orbital limit points fails for $F^*$.
\end{proof}

\begin{example}\label{example} Let $1<p<\infty$, and $|\lambda|=1$. Set
$
a_{n}=1 ~\text{and}~ b_{n}=\frac{-1}{\lambda} \frac{n+1}{n+2}, ~n\geq 0.$ Then, $\frac{1}{R}=\limsup_{n\rightarrow \infty}\left(1+\frac{n+1}{n+2}\right)^{1/n}=1,$
and 
\[
\|z^n\|_{\ell^p_{a,b}}^p:=\frac{1}{|a_{n}|^{p}}\left(1+
\sum_{j=1}^{\infty}
\left|\prod_{k=0}^{j-1}\frac{b_{n+k}}{a_{n+k+1}}
\right|^{p}\right)=1+(n+1)^p\sum_{j=2}^{\infty}\frac{1}{(n+j)^p}.
\]
By (\ref{m}), we have that $z^n \in \ell^p_{a,b}.$

We now observe that $F$ is bounded. Indeed, the matrix of $F$ with respect to the basis $\{f_n\}_{n=0}^{\infty}$ is given by
\begin{eqnarray*}
     [F]&:=&\begin{bmatrix} 0 & 0 & 0&0&0\cdots\\ 1 & 0 & 0&0&0\ddots\\ \frac{1}{\lambda}\frac{1}{2.3} & 1 & 0&0&0\ddots\\ \frac{1}{\lambda^{2}}\frac{1}{2.4} & \frac{1}{\lambda}\frac{1}{3.4} & 1&0&0\ddots\\ \frac{1}{\lambda^{3}}\frac{1}{2.5} &\frac{1}{\lambda^{2}}\frac{1}{3.5} & \frac{1}{\lambda}\frac{1}{4.5}&1&0\ddots\\ \vdots&\vdots &\ddots&\ddots&\ddots\ddots \end{bmatrix}.
\end{eqnarray*}
Hence, $F$ is the sum of the shift $S$ and some operator $K$ on $\ell^p(\mathbb{N}_0)$. Since the entries in the matrix of $K$ are $p$-absolutely summable, $p>1$, we get that $K$ is continuous. (For details, we refer to M. Koskela \cite{Koskela}.) When $p=2$, $K$ is in fact a Hilbert-Schdmidt operator since the entries in the matrix of $K$ are square summable. (Recall that an operator $T$ on a separable Hilbert space $H$ is Hilbert-Schmidt if $\sum_n \|Te_n\|^2<\infty$ for some orthonormal basis $\{e_n\}$ in $H$, cf. \cite{Halmos}.) 

Also, $F^*$ has the properties stated in (i), (ii), and (iii) of the above theorem. Note that $(A1)$ and $(A2)$ are satisfied. Indeed, a computation shows that
\[
\sum_{n\geq 0} |1+\lambda b_n|^q=\sum_{n\geq 0} \frac{1}{(n+2)^q},
\]
which is finite, and 
\[
 \sup_{n\geq 0}\left(|1+\lambda b_{n}|^{p}\Big(\sum_{j=1}^{\infty}
\Big|\prod_{k=1}^{j} b_{n+k}
\Big|^{p}\Big)\right)= \sup_{n\geq 0}~\sum_{j\geq 3}\frac{1}{(n+j)^{p}}<\infty.
\]
The requirement in (A3) holds because of the following argument. If
\[
f(z)=\sum_{n\geq 0}\lambda_n \Big(1-\frac{1}{\lambda}\frac{n+1}{n+2}z\Big)z^n
\]
is a member in $\ell^p_{a,b}$ with $\{\lambda_n\}\in \ell^p(\mathbb{N}_0)$, then this series converges (radially) to $f(\lambda)$, uniformly in some interval $(r_0,1]$.
\end{example}

In view of the above results, we have:

\begin{corollary}
There exists a space of the form $\ell^p_{a,b}$, $1<p<\infty$, such that the shift $F$ is bounded and non-similar to a weighted forward shift on $\ell^p(\mathbb{N}_0)$.
\end{corollary}

 \begin{remark}
      By Example \ref{example}, there is a compact operator on $\ell^2(\mathbb{N}_{0})$ such that the perturbation $B+K$ has periodic vectors, but it is non-hypercyclic. Also, such a perturbation exists for which the zero-one law fails. We would like to mention that compact operators with these properties can be easily obtained, without appealing to Example \ref{example}. Indeed, for $|\lambda|=1$, consider the one-rank operator $K_{\lambda}$ on $\ell^p(\mathbb{N}_{0})$, defined by $K_{\lambda}(x):=\lambda x_0e_0$, for $x=(x_0,x_1,\ldots)$. Then, $(B+K_{\lambda})e_0=\lambda e_0$, and so $B+K_{\lambda}$ has an eigenvalue, namely $\lambda$ with an eigenvector $e_0$. Hence, arguing as in the proofs of  (ii) and (iii) of Theorem \ref{pr}, we get the same conclusions for $B+K_{\lambda}$. The perturbation $B+K_{\lambda}$ does not have hypercyclic vectors: if $x=(x_m)_{m=0}^{\infty}\in\ell^p(\mathbb{N}_{0})$, then we have 
      \[
      (B+K_{\lambda})^n(x)=(x_n+\lambda x_{n-1}+\cdots+\lambda^n x_0, x_{n+1},x_{n+2},\cdots),
      \]
      which cannot form a dense set as the projection along the second coordinate converges to $0$. It would be interesting to study various dynamical aspects of compact perturbations of weighted backward shifts on $\ell^p(\mathbb{N}_{0})$. We refer to \cite{Herrero} for compact perturbations of hypercyclic operators, and \cite{Stampfli} for perturbations results on a shift.
  \end{remark}
Next, we show that, under suitable extra assumptions, the zero-one law of hypercyclicity holds. We will consider the adjoint $(F_w^{*})^{\nu}$ as a matrix transformation on $\ell^q(\mathbb{N}_{0})$, where $1/p+1/q=1$, by finding the matrix representation of $(F_w^{*})^{\nu}$. For $\nu\geq 1$ and $n\geq 0$, if we set
    \[
    c_{n,\nu}:=w_{n+1}\cdots w_{n+\nu}\frac{b_{n}}{a_{n+\nu+1}}-w_{n}\cdots w_{n+\nu-1}\frac{a_{n}b_{n+\nu}}{a_{n+\nu}a_{n+\nu+1}}~~~~\text{and}~~~~A_{n,\nu}:=w_{n}\cdots w_{n+\nu-1}\frac{a_{n}}{a_{n+\nu}},
    \]
    then by the action of $F_{w}^{\nu}$ on the normalized Schauder basis $f_{n}(z) =(a_{n}+b_{n}z)z^{n},n \geq 0$, of $\ell^{p}_{a,b}$, we obtain 
    \begin{eqnarray*}
F_{w}^{\nu}(f_{n})(z)&=&w_{n}\cdots w_{n+\nu-1}a_{n}z^{n+\nu}+w_{n+1}\cdots w_{n+\nu}b_{n}z^{n+\nu+1}\\
&=&\frac{w_{n}\cdots w_{n+\nu-1}a_{n}}{a_{n+\nu}}f_{{n+\nu}} + \left(w_{n+1}\cdots w_{n+\nu}b_{n}-\frac{w_{n}\cdots w_{n+\nu-1}a_{n}b_{n+\nu}}{a_{n+\nu}}\right) z^{n+\nu+1}\\
&=&A_{n,\nu}f_{{n+\nu}}+c_{n,\nu}a_{n+\nu+1}z^{n+\nu+1},
\end{eqnarray*}
for $\nu\geq 1$. Now, using the expansion of $z^{n+\nu+1}$ from (\ref{monomial-expansion}), we get
\[
F_{w}^{\nu}(f_{n})(z)=A_{n,\nu}f_{{n+\nu}}+c_{n,\nu}\sum_{j=0}^{\infty}(-1)^{j}\left(\Mprod_{k=0}^{j-1}\frac{b_{n+\nu+1+k}}{a_{n+\nu+2+k}}\right)f_{n+\nu+1+j},
\]
where the term corresponding to $j=0$ in the sum is understood as $1$. Therefore, the matrix of $F_{w}^{\nu}$ with respect to the normalized Schauder basis $\{f_{n}\}_{n\geq 0}$ is given by
$$[F_{w}^{\nu}]:=\begin{bmatrix}
\cdots &\cdots & \cdots&\cdots&\cdots\\
&\nu~~~ rows~~~ of ~~zeros& \\
\cdots &\cdots & \cdots&\cdots&\cdots\\
A_{0,\nu} & 0 & 0&0&\ddots\\
c_{0,\nu} & A_{1,\nu}  & 0&0&\ddots\\
-c_{0,\nu}\frac{b_{\nu+1}}{a_{\nu+2}} & c_{1,\nu} & A_{2,\nu} &0&\ddots\\
 c_{0,\nu}\frac{b_{\nu+1}b_{\nu+2}}{a_{\nu+2}a_{\nu+3}}& -c_{1,\nu}\frac{b_{\nu+2}}{a_{\nu+3}} & c_{2,\nu}&A_{3,\nu}&\ddots\\
-c_{0,\nu}\frac{b_{\nu+1}b_{\nu+2}b_{\nu+3}}{a_{\nu+2}a_{\nu+3}a_{\nu+4}} & c_{1,\nu}\frac{b_{\nu+2}b_{\nu+3}}{a_{\nu+3}a_{\nu+4}} &-c_{2,\nu}\frac{b_{\nu+3}}{a_{\nu+4}} &c_{3,\nu}&\ddots\\
\vdots &\ddots&\ddots&\ddots&\ddots
\end{bmatrix},$$
from which it follows that the matrix of $(F_{w}^{*})^{\nu}$ is
\begin{equation}\label{adoint-matrix}
    [(F_{w}^{*})^{\nu}]:=\begin{bmatrix}
\vdots&\vdots&\vdots&A_{0,\nu} & c_{0,\nu} & -c_{0,\nu}\frac{b_{\nu+1}}{a_{\nu+2}}&c_{0,\nu}\frac{b_{\nu+1}b_{\nu+2}}{a_{\nu+2}a_{\nu+3}}&\ddots\\
\vdots&\nu&\vdots& 0& A_{1,\nu}  & c_{1,\nu} &-c_{1,\nu}\frac{b_{\nu+2}}{a_{\nu+3}}&\ddots\\
\vdots&columns&\vdots&0 & 0& A_{2,\nu} &c_{2,\nu}&\ddots\\
 \vdots&of&\vdots&0& 0 & 0&A_{3,\nu}&\ddots\\
\vdots&zeros&\vdots&0 & 0 &0&0&\ddots\\
\vdots&\vdots&\vdots&\vdots &\ddots&\ddots&\ddots&\ddots
\end{bmatrix}.
\end{equation}

With this matrix transformation, we have a zero-one law for the hypercyclicity of $F_w^*$. Recall that, in Bonilla et al. \cite{Bonilla-Cardeccia} a similar dichotomy was obtained for unilateral weighted backward shifts on certain Fr\'{e}chet sequence spaces including $\ell^p(\mathbb{N}_0)$. In what follows, $\mathbb{C}^{\mathbb{N}_0}$ denotes the space of all complex sequences $(\lambda_0,\lambda_1,\cdots)$, equipped with the topology of coordinate-wise convergence.

\begin{theorem}\label{T}
    Assume that for $1<p<\infty$ and $\frac{1}{p}+\frac{1}{q}=1,$
    \begin{equation}\label{yy}
\sup_{n\geq 0}~\left\lvert\frac{ w_{n}a_{n}}{a_{n+1}}\right\rvert<\infty, \hspace{.2cm} and \hspace{.4cm} \lim \sup_{n\rightarrow \infty}\left \lvert \frac{b_{n}}{a_{n+1}}\right \rvert<1.
\end{equation}\label{d}
Then, the following statements are equivalent:
      \begin{itemize}
          \item[(i)] $F_{w}^{*}$ is not hypercyclic on $(\ell^p_{a,b})^*.$
             \item[(ii)] For each $u \in \ell^q(\mathbb{N}_0),$ one has that $[F_{w}^{*}]^{\nu}u\to 0$ in $\mathbb{C}^{\mathbb{N}_0}$ as $\nu\to \infty,$ where $u$ is regarded as a column vector.
      \end{itemize}
\end{theorem}

\begin{proof}
Note that, by the first two conditions of the theorem, $F_w$ becomes a continuous operator, and also, the space $\ell^p_{a,b}$ contains the polynomials. For $\nu\geq 1$, we recall the matrix of $(F_{w}^{*})^{\nu}$ from (\ref{adoint-matrix}). Then, for $u=(\lambda_{n})_{n\geq 0}\in \ell^{q}(\mathbb{N}_{0}),$ we get
\begin{equation}\label{cc}
   [(F_{w}^{*})^{\nu}]u=\sum_{n=0}^{\infty} \alpha_{n,\nu} ~ e_{n}.
\end{equation}
where $\{e_{n}\}_{n\geq 0}$ is the standard basis in $\ell^q(\mathbb{N}_{0})$, and 
\[
\alpha_{n,\nu}= \lambda_{\nu+n} A_{n,\nu} +\lambda_{\nu+n+1}c_{n,\nu}+c_{n,\nu}\displaystyle{\sum_{j=2}^{\infty}(-)^{j-1}\lambda_{\nu+n+j}\frac{b_{\nu+n+1}\cdots b_{\nu+n+j-1}}{a_{\nu+n+2}\cdots a_{\nu+n+j}}}.
\]
Now, assume that $F_{w}^{*}$ is not hypercyclic on $(\ell^p_{a,b})^*$. By Theorem \ref{unilateral}, we have 
\[
M:=\sup_{n\geq 1}\left|\frac{w_{0}\cdots w_{n-1}}{a_{n}}\right|<\infty.
\]
We will show that $\lim_{\nu \rightarrow \infty}\alpha_{n,\nu}=0$, for each fixed $n\geq 0$, by looking at the three terms in $\alpha_{n,\nu}$. For this, note that 
\begin{center}
$A_{n,\nu}=w_{n}\cdots w_{n+\nu-1}\frac{a_{n}}{a_{n+\nu}}=\frac{w_{0} \cdots w_{n+\nu-1}}{a_{n+\nu}}\frac{a_{n}}{w_{0}\cdots w_{n-1}}$, 
\end{center}
and hence, $\{A_{n,\nu}:\nu\geq 1\}$ is bounded. Also, since $\limsup_n |b_n/a_{n+1}|<1$, we can find some $r<1$ and $N\geq 1$ such that 
\begin{equation}\label{dd}
|b_n/a_{n+1}|<r,
\end{equation}
for all $n\geq N$. We now see that $\{c_{n,\nu}:\nu\geq 1\}$ is bounded for each fixed $n\geq 0$, from the following observation:
\begin{eqnarray*}
   |c_{n,\nu}|&=&\left|w_{n+1}\cdots w_{n+\nu}\frac{b_{n}}{a_{n+\nu+1}}-w_{n}\cdots w_{n+\nu-1}\frac{a_{n}b_{n+\nu}}{a_{n+\nu}a_{n+\nu+1}}\right| \\
   &\leq& \left|\frac{w_{0}\cdots w_{n+\nu}}{a_{n+\nu+1}}\right|\left|\frac{a_{n+1}}{w_{0}\cdots w_{n}}\right|\left|\frac{b_{n}}{a_{n+1}}\right|+\left|\frac{w_{0}\cdots w_{n+\nu-1}}{a_{n+\nu}}\right|\left|\frac{a_{n}}{w_{0}\cdots w_{n-1}}\right|\left|\frac{b_{n+\nu}}{a_{n+\nu+1}}\right|.
\end{eqnarray*}
Using \eqref{dd} and H\"{o}lder inequality, we also obtain that, for large $\nu$
\[
\left|\sum_{j=2}^{\infty}(-)^{j-1}\lambda_{\nu+n+j}\frac{b_{\nu+n+1}\cdots b_{\nu+n+j-1}}{a_{\nu+n+2}\cdots a_{\nu+n+j}}\right|\leq C \left(\sum_{j\geq \nu}|\lambda_j|^q\right)^{1/q},
\]
where $C$ is a constant depending only on $r$. Combining all these inequalities, we can conclude that the coefficients in \eqref{cc} converge to $0$, as $\nu\rightarrow \infty$.
\end{proof} 

Now, we obtain the zero-one law of hypercyclicity for $F_w^*$ in view of the above theorem.
\begin{proposition}
Under the assumptions \eqref{yy} of Theorem \ref{T}, the following are equivalent.
 \begin{itemize}
          \item[(i)] $F_{w}^{*}$ is hypercyclic on $(\ell^p_{a,b})^*.$
             \item[(ii)] $F_w^*$ has an orbit admitting a non-zero weak sequential limit point.
             \item[(iii)] $F_w^*$ has an orbit with a non-zero limit point in the norm topology.
      \end{itemize}
\end{proposition}

{\bf Acknowledgments.} We thank the referees for a careful reading of our manuscript and providing a number of suggestions which now helped to improve the overall presentation of the paper. B.K. Das is supported financially by CSIR research fellowship (File No.: 09/1059(0037)/2020-EMR-I), and A. Mundayadan is partially funded by a Start-Up Research Grant of SERB-DST (File. No.: SRG/2021/002418).

\bibliographystyle{amsplain}

\end{document}